\documentclass[11pt,psamsfonts]{amsart}
\usepackage[mathscr]{eucal}
\usepackage{upref}

\usepackage{verbatim}
\usepackage{color}

\theoremstyle{plain}
\newtheorem{theorem}{Theorem}
\newtheorem{lemma}{Lemma}
\newtheorem{corollary}{Corollary}

\theoremstyle{definition}

\theoremstyle{remark}
\newtheorem{remark}{Remark}
\newtheorem{example}{Example}

\numberwithin{equation}{section}

 \newcommand{\ol}{\overline}
 
 \newcommand{\wh}{\widehat}
 \newcommand{\wt}{\widetilde}
 \newcommand{\wb}{\overline}
 \newcommand{\upwt}{\widetilde{\phantom{a}}}

 \newcommand{\scra}{\mathscr{A}}
\newcommand{\scrb}{\mathscr{B}}
\newcommand{\scrk}{\mathscr{K}}

 \newcommand{\setC}{\mathbb{C}}
 \newcommand{\setN}{\mathbb{N}}
 \newcommand{\setR}{\mathbb{R}}

 \newcommand{\ve}{\varepsilon}
 \newcommand{\vr}{\varrho}

 \DeclareMathOperator{\id}{id}
 \DeclareMathOperator{\re}{Re}
 \DeclareMathOperator{\imp}{Im}
 \DeclareMathOperator{\pa}{\partial}
 \DeclareMathOperator{\GL}{GL}

\begin{document}

\title[Equations in Dirichlet series and Laplace transforms]
{General Dirichlet series, arithmetic\\[1ex]
convolution equations and\\[1ex]
Laplace transforms}

\author{Helge Gl\"{o}ckner}
\author{Lutz G. Lucht}
\author{\v{S}tefan Porubsk\'{y}}

\thanks{The first author was supported by the Deutsche
Forschungsgemeinschaft, grants GL 357/2-1
and GL 357/5-1.}

\thanks{The second author is indebted to the Grant Agency of the Czech 
Republic, Grant \#~201/07/0191, for support.}

\thanks{The third author was supported by the Grant Agency of the Czech
Republic, Grant~\#~201/07/0191, and by the Institutional Research
Plan AV0Z10300504. For some of the period during which this work
was carried out he also thanks the Clausthal University of
Technology for support.}

\address{Institut f\"{u}r Mathematik\\
Universit\"{a}t Paderborn\\
Warburger Str.\ 100\\
33098 Paderborn\\
Germany}
\email{glockner@math.upb.de}

\address{Institute of Mathematics\\
Clausthal University of Technology\\
Erzstr. 1\\
38678 Clausthal-Zellerfeld\\
Germany}
\email{lucht@math.tu-clausthal.de}

\address{Institute of Computer Science\\
Academy of Sciences of the Czech Republic\\
Pod Vod\'{a}renskou v\v{e}\v{z}\'{\i} 2\\
182~07 Prague 8\\
Czech Republic}
\email{Stefan.Porubsky@cs.cas.cz}




\date{December 4, 2007}


\begin{abstract}
In the earlier paper \cite{GLP2006}, we studied solutions
$g\colon\setN\to\setC$ to convolution equations of the form
\[
a_d*g^{*d}+a_{d-1}*g^{*(d-1)}+\cdots+a_1*g+a_0=0,
\]
where $a_0,\ldots,a_d\colon\setN\to\setC$ are given arithmetic functions
associated with Dirichlet series which converge on some right half plane, and
also $g$ is required to be such a function. In this article, we extend
our previous results to multidimensional general Dirichlet series of the
form $\sum_{x\in X}f(x)e^{-sx}$ ($s\in \setC^k$), where
$X\subseteq [0,\infty)^k$ is an additive subsemigroup. If $X$ is discrete and
a certain solvability criterion is satisfied, we determine solutions by an
elementary recursive approach, adapting an idea of Fe\v{c}kan \cite{Fec2006}.
The solution of the general case leads us to a more comprehensive question: 
Let $X$ be an additive subsemigroup of a pointed, closed convex cone
$C\subseteq \setR^k$. Can we find a complex Radon measure on $X$ whose
Laplace transform satisfies a given polynomial equation whose coefficients
are Laplace transforms of such measures\,?
\end{abstract}

\maketitle


\section{Results on general Dirichlet series and Laplace transforms}
\label{S1}

Let $k\in\setN$ and $X\subseteq[0,\infty)^k$ be an infinite additive 
semigroup of $k$-tuples $x=(x_1,\ldots,x_k)$ with $0=(0,\ldots,0)\in X$. 
Assume until further notice that $X$ is discrete (i.e., without a cluster 
point in $\setR^k$) and hence countable. In this situation, the set 
$\scra=\scra(X):=\setC^X$ of all arithmetic functions $g\colon X\to\setC$ is 
a commutative complex algebra, called \emph{Dirichlet algebra of $X$}, under 
the usual linear operations and the \emph{convolution} 
$*\colon\scra^2\to\scra$ as the algebra multiplication which for arbitrary 
functions $g,h\in\scra$ is defined by
\begin{equation*}
 \big(g*h\big)(x)=
 \sum_{\substack{x',\,x''\in X \\ x'+x''=x}}\,g(x')\,h(x'')
 \qquad(x\in X).
\end{equation*}
The unit element $u\in\scra$ under $*$ is given by $u(0)=1$ and $u(x)=0$ for
non-zero $x\in X$, and the \emph{multiplicative group} $\scra^*$ of $\scra$
with respect to the convolution consists of all functions $g\in\scra$
satisfying $g(0)\neq 0$. For brevity, we write $g^{*j}$ for the convolution 
$g*\cdots*g$ with $j$ factors $g\in\scra$, $g^{*0}:=u$, and $g^{-1}$ for the 
inverse of $g$\,, i.e., $g*g^{-1}=u$.

Given $g\in\scra$, consider the \emph{$k$-dimensional Dirichlet series}
\begin{equation} \label{eq1.1}
 \wt{g}(s)=\sum_{x\in X}\,g(x)\,e^{-x\cdot s} \qquad(s\in\setC^k)
\end{equation}
with the inner%
\footnote{Because $x$ is real, we omit complex conjugation.} 
product $x\cdot s=x_1 s_1+\cdots+x_k s_k$ for $x=(x_1,\ldots,x_k)\in X$ and 
$s=(s_1,\ldots,s_k)\in\setC^k$.  For fixed $X$, these multidimensional series 
form an algebra which is isomorphic to $\scra$ under $~\wt{}:g\mapsto\wt{g}$ 
for $\wt{g}(s)\cdot\wt{h}(s):=(g*h)\upwt(s)$.

The aim of this paper is to investigate existence and analytic behavior of
the solutions $g\in\scra$ to the equation $Tg=0$ for convolution
polynomials $T\colon\scra\to\scra$ of degree $d\in\setN$ defined by
\begin{equation} \label{poleq}
Tg=a_d*g^{*d}+a_{d-1}*g^{*(d-1)}+\cdots+a_1*g+a_0
\end{equation}
with given arithmetic functions $a_d,a_{d-1},\ldots,a_1,a_0\in\scra$,
$a_d\neq 0$.

Our first theorem concerns the existence of solutions $g\in\scra$ to $Tg=0$.

\begin{theorem} \label{T1}
For $T$ as in \eqref{poleq}, let the polynomial $f(z)\in\setC[z]$ be
defined by
\begin{equation} \label{initpol}
f(z)=a_d(0)\,z^d+a_{d-1}(0)\,z^{d-1}+\cdots+a_1(0)\,z+a_0(0).
\end{equation}
If $z_0$ is a simple zero of $f(z)$, then there exists a uniquely determined
solution $g\in\scra$ to $Tg=0$ satisfying $g(0)=z_0$. If $f(z)$ has no simple
zeros, then $Tg=0$ need not possess any solution. In any case $Tg=0$ has at
most $d$ solutions $g\in\scra$.
\end{theorem}

Section \ref{S2} contains an elementary proof extending that of 
\cite{GLP2006} from $\scra=\scra(\log{\setN})$ to $\scra(X)$ with
$X\subseteq [0,\infty)^k$ a general discrete additive semigroup.

Given $x=(x_1,\ldots, x_k)\in \setR^k$, define its \emph{size} as 
$|x|:=|x_1|+\cdots+|x_k|$. For the study of absolute convergence of Dirichlet 
series $\wt{g}(s)$ it is convenient to consider the Banach algebras $\scra_r$ 
of arithmetic functions $g\in\scra$ with bounded \emph{$r$-norm}
\begin{equation*}
\|g\|_r:=\sum_{x\in X}\,|g(x)|\,e^{-r|x|}
\end{equation*}
for $r\in\setR$, and also the algebra
\begin{equation*}
\scra_\infty:=\bigcup\,\big\{\scra_r\colon r\in\setR\big\}.
\end{equation*}
Let $H_r:=\big\{s\in\setC:\re{s}>r\big\}$ and $\wb{H}_r$ be its
closure. Given $g\in\scra_r$, the series $\wt{g}(s)$ converges
absolutely for $s\in\wb{H}_r^{\,k}$. We shall use the same symbol,
$\wt{g}(s)$, also for the value of the sum.

The next theorem guarantees that under suitable conditions solutions
$g\in\scra$ to $Tg=0$ belong to $\scra_\infty$, if
$a_d,\ldots,a_0\in\scra_\infty$.

\begin{theorem} \label{T2}
Suppose that $a_d,a_{d-1},\ldots,a_1,a_0\in\scra_\infty$ in \eqref{poleq}, 
with $a_d\neq 0$. If $z_0$ is a simple zero of $f(z)$ in \eqref{initpol}, 
then there exists a solution $g\in\scra_\infty$ to $Tg=0$ satisfying
$g(0)=z_0$.
\end{theorem}

In Section \ref{S3} we give an elementary proof of Theorem \ref{T2} 
extending that of Fe\v{c}kan \cite{Fec2006} in the case 
$\scra=\scra(\log{\setN})$ to the multidimensional case $\scra(X)$.

As in \cite{GLP2006}, Theorem \ref{T2} is the special case
\begin{equation*}
 m=1,\quad
 F\colon\setC^{d+1}\times\setC\to\setC,\quad
 F(w_0,\ldots,w_d,z)=\sum_{j=0}^d\,w_j\,z^j
\end{equation*}
of a multidimensional version:

\begin{theorem} \label{T3}
For open subsets $V\subseteq\setC^n$ and $Z\subseteq\setC^m$, let
\begin{equation*}
F\colon V\times Z\to\setC^m,\quad (v,z)\mapsto F(v,z)
\end{equation*}
be a holomorphic function and $(v_0,z_0)\in V\times Z$ be such that 
$F(v_0,z_0)=0$. Let $a_1,\ldots, a_n\in\scra=\scra(X)$ satisfy the condition 
$\big(a_1(0),\ldots,a_n(0)\big)=v_0$. If the differential $\pa_z{F(v_0,z)}$ 
at $z=z_0$ is in $\GL_m(\setC)$, then there exists a unique $m$-tuple 
$(g_1,\ldots,g_m)\in\scra^m$ such that
\begin{align}
 \big(g_1(0),\ldots,g_m(0)\big)   &\,=\,z_0\quad\text{and} \label{cond1} \\
 F[a_1,\ldots,a_n,g_1,\ldots,g_m] &\,=\,0. \label{eq1.5}
\end{align}
If, in addition, $a_1,\ldots,a_n\in\scra_\infty$, then also
$g_1,\ldots,g_m\in\scra_\infty$ and both
\begin{equation}\label{condpre2}
 \big(\wt{a}_1(s),\ldots,\wt{a}_n(s),\wt{g}_1(s),\ldots,\wt{g}_m(s)\big)
 \in V\times Z
\end{equation}
and
\begin{equation} \label{cond2}
 F\big(\wt{a}_1(s),\ldots,\wt{a}_n(s),
 \wt{g}_1(s),\ldots,\wt{g}_m(s)\big)\,=\,0
\end{equation}
hold for all $s\in\ol{H}_r^{\,k}$, with $r$ sufficiently large. In this case,
also \eqref{cond1}, \eqref{condpre2} and \eqref{cond2} uniquely determine
$(g_1,\ldots,g_m)\in\scra_\infty^m$.
\end{theorem}

Here the left hand side of \eqref{eq1.5} is obtained via multivariable
holomorphic functional calculus (see \cite{Bil2004} for a recent exposition
in the required generality).

In Section \ref{S4} a proof of Theorem \ref{T3} is given based on a version 
of the implicit function theorem proved by Biller \cite{Bil2004} and 
techniques taken from the theory of commutative topological algebras.

The results concerning generalized Dirichlet series are special cases of
results concerning Laplace transforms (proved in Section \ref{S6}), which we
outline now. The necessary background material concerning convex cones and 
Laplace transforms (of positive measures) needed here can be found, e.g., in 
\cite{Nee2000} or \cite{Glo2003}.

\noindent
In the sequel, let $C\subseteq\setR^k$ be a non-empty closed \emph{convex 
cone}%
\footnotemark[1]. Assume that $C$ is \emph{pointed}%
\footnotemark[2] and \emph{generating}%
\footnotemark[3].
\footnotetext[1]{That is, $C$ is convex and $[0,\infty)\cdot C\subseteq C$.}%
\footnotetext[2]{That is, $C\cap (-C)=\{0\}$ or, equivalently, $C$ contains 
no lines \cite[Corollary~V.1.11]{Nee2000}.}%
\footnotetext[3]{That is, $C-C=\setR^k$ or, equivalently, $C$ has non-empty
interior $C^\circ$ (cf. \cite[Proposition 5.1.4\,(ii)]{Nee2000}).}%
Since $C$ is pointed, the dual cone 
$C^\star:=\{y\in\setR^k:xy\geq 0\text{ for all }x\in C\}$ has non-empty 
interior (see \cite[Proposition V.1.5\,(ii)]{Nee2000}). We pick a  
$y_0\in (C^\star)^\circ$. Then $xy_0>0$ for each $x\in C\setminus\{0\}$
(cf.~\cite[Proposition V.1.4\,(v)]{Nee2000}).

Let $X\subseteq C$ be an arbitrary (not necessarily discrete or countable) 
additive subsemigroup with $0\in X$, equipped with a Hausdorff topology which 
makes the inclusion map $X\to \setR^k$ continuous and turns $X$ into an 
additive topological semigroup (i.e., addition $X\times X\to X$ is 
continuous).             

Consider the complex vector space $M(X)$ of all complex (not necessarily 
bounded) Radon measures $\mu$ on $X$, defined on the $\delta$-ring 
$\delta(X)$ generated by the set of compact subsets of $X$ (see Section 
\ref{S5} for our measure-theoretic setting). For $\mu\in M(X)$, let 
$|\mu|\in M_+(X)$ denote the associated total variation measure. Given
$r\in\setR$, let $M_r(X)$ be the set of all $\mu\in M(X)$ such that
\[
\|\mu\|_r := \int_X e^{-rxy_0}\,d|\mu|(x)<\infty\,.
\]
To emphasize the dependence on $y_0$, we occasionally write $M_r^{(y_0)}(X)$ 
instead of $M_r(X)$ and $\|\mu\|_r^{(y_0)}$ instead of $\|\mu\|_r$. Then 
$M_r(X)$ is a vector subspace of $M(X)$ and the convolution of measures turns 
$M_r(X)$ into a complex algebra. In fact, $(M_r(X),\|\cdot\|_r)$ is a 
commutative Banach algebra with unit element $\delta_0$ (point mass at
$0\in X$). It is isomorphic to the Banach algebra $M_0(X)$ of bounded complex 
Radon measures on $X$ via $\mu\mapsto e^{-rxy_0}d\mu(x)$, where we write 
$f(x)\,d\mu(x)$ or $f\odot\mu$ for the measure of density $f$ with respect 
to~$\mu$.

From $M_r(X)\subseteq M_t(X)$ for $r\leq t$ we conclude that 
$M_\infty(X):=\bigcup_{r\in \setR}M_r(X)$ is an algebra under convolution and 
is independent of the choice of $y_0$. Given $\mu\in M_r(X)$, the set
\[
D_\mu:=\left\{s\in\setC^k:\int_X e^{-x\re{s}}\,d|\mu|(x)<\infty\right\}
\]
is convex (cf. \cite[Proposition~V.4.3]{Nee2000}), where $\re{s}$ stands for 
the real part of $s$. We define the \emph{Fourier-Laplace transform} 
$\wt{\mu}\colon D_\mu\to\setC$ of $\mu$ by
\[
\wt{\mu}(s):=\int_X e^{-x\cdot s}\,d\mu(x)\,.
\]
Then $ry_0+C^\star+i\setR^k\subseteq D_\mu$, entailing that the convex set
$D_\mu\subseteq\setC^k$ has non-empty interior. Given $\mu\in M_r(X)$, we 
write
\[
\wt{\mu}(\infty):=\mu(\{0\})\,.
\]
For $r\in\setR$ and a closed unital subalgebra $A\subseteq M_0(X)$, we set 
$A_r:=\{e^{rxy_0}d\mu(x)\colon\mu\in A\}$ and define 
$A_\infty:=\bigcup_{r\in \setR}A_r$. Typically, $A=M_0(X)$.

\begin{theorem} \label{Tnew}
Let $V\subseteq\setC^n$ and $Z\subseteq\setC^m$ be open subsets, the function
\begin{equation*}
F\colon V\times Z\to\setC^m,\quad (v,z)\mapsto F(v,z)
\end{equation*}
be holomorphic, and $(v_0,z_0)\in V\times Z$ with $F(v_0,z_0)=0$. Let 
$\lambda_1,\ldots, \lambda_n\in A_\infty$ be measures such that
$\big(\wt{\lambda}_1(\infty),\ldots,\wt{\lambda}_n(\infty)\big)=v_0$.
If the differential $\pa_z{F(v_0,z)}$ at $z=z_0$ is in $\GL_m(\setC)$, then 
there exists a unique $m$-tuple $\mu=(\mu_1,\ldots,\mu_m)\in A_\infty^m$ such 
that
\begin{equation}\label{cond0new}
\big(\wt{\mu}_1(\infty),\ldots,\wt{\mu}_m(\infty)\big)
\,=\,z_0
\end{equation}
and, for some $r\in\setR$,
\begin{equation}\label{cond1new}
 \big(\wt{\lambda}_1(s),\ldots,\wt{\lambda}_n(s),
 \wt{\mu}_1(s),\ldots,\wt{\mu}_m(s)\big)\in V\times Z
\end{equation}
holds for all $s\in ry_0 +C^\star+i\setR^k$ as well as
\begin{equation}\label{cond2new}
F\big(\wt{\lambda}_1(s),\ldots,\wt{\lambda}_n(s),
\wt{\mu}_1(s),\ldots,\wt{\mu}_m(s)\big)\,=\,0.
\end{equation}
Here $\mu$ is also determined by \eqref{cond0new} and the condition that
\begin{equation}\label{cond3new}
F[\lambda_1,\ldots,\lambda_n,\mu_1,\ldots,\mu_m]=0
\end{equation}
holds in the commutative Banach algebra $A_r$ for some $r\in\setR$ such
that $\lambda_1,\ldots,\lambda_n$, $\mu_1,\ldots,\mu_m\in A_r$.
\end{theorem}

The preceding theorem subsumes the following special cases.

\begin{example} \label{E1}
As before, let $C\subseteq\setR^k$ be a pointed and generating closed convex 
cone. Let $X\subseteq C$ be an arbitrary subsemigroup such that $0\in X$, and 
endow $X$ with the discrete topology. For fixed $y_0\in (C^\star)^\circ$ and 
$r\in\setR$, define the weight $w\colon X\to\setR_+$ by $w(x)=e^{-rxy_0}$ and 
the $r$-norm $\|f\|_r$ of $f\colon X\to\setC$ by
\begin{equation*}
\|f\|_r:=\sum_{x\in X}\,w(x)\,|f(x)|\,,
\end{equation*}
where the right hand side denotes the supremum in $[0,\infty]$ of the finite 
partial sums. Then $\ell_w^1(X):=\{f\in\setC^X:\|f\|_r<\infty\}$ is the 
complex vector space of functions of finite norm, and Theorem \ref{Tnew} 
applies to $X$. Note that the map                
\begin{equation*}
M(X)\to\setC^X\,,\quad \mu\mapsto f_\mu
\end{equation*}
with $f_\mu(x):=\mu(\{x\})$ is an isomorphism of vector spaces which induces
an isomorphism $M_r(X)\cong\ell^1_w(X)$. Given $\mu$ and $f=f_\mu$ as
before, we have
\begin{equation*}
\|\mu\|_r=\sum_{x\in X}\,|f(x)|\,e^{-rxy_0}=\|f\|_r\,.
\end{equation*}
If $\mu\in M_r(X)$, then $\|\mu\|_r<\infty$ and hence $f(x)=0$ for all but
countably many $x$ (cf. \cite[4.15]{Rud1987}). Then
\begin{equation}\label{corresp}
\wt{\mu}(s)=\wt{f}(s):=\sum_{x\in X}\,f(x)\,e^{-sx}
\end{equation}
for all $s\in D_\mu$, where the right hand side is interpreted as the limit
of an absolutely summable family of vectors
(cf. \cite[Chapter V, \S3]{Die1960}).
\end{example}

\begin{example} \label{E2}
If $C=[0,\infty)^k$ in Example \ref{E1}, then $C^\star=[0,\infty)^k$ with 
$y_0:=(1,1,\ldots,1)$ in its interior and $ry_0+C^\star+i\setR^k=\ol{H}_r^k$ 
for each $r\in\setR$. Equipping an arbitrary subsemigroup 
$X\subseteq [0,\infty)^k$ containing $0$ with the discrete topology, and 
taking $A:=M_0(X)$, Theorem \ref{Tnew} provides a generalization of the 
second half of Theorem \ref{T3}, which does not require discreteness of $X$ 
(nor countability).
\end{example}

\begin{example} \label{E3}
Theorem \ref{Tnew} also applies to an arbitrary additive semigroup 
$X\subseteq C$ endowed with the topology induced by $\setR^k$. For example, 
we can let $X:=C$ be a pointed and generating closed convex cone in 
$\setR^k$, with the induced topology.
\end{example}

\begin{example} \label{E4}
If $X\subseteq C$ is a Borel measurable subset of $\setR^k$ (e.g., if $X=C$),
we can equip $X$ with the topology induced by $\setR^k$ and choose
$A:=\setC\delta_0+L^1(X,\lambda)\odot\lambda$, where $\lambda$ denotes the
restriction of Lebesgue-Borel measure on $\setR^k$ to $\delta(X)$.
\end{example}

\section{Proof of Theorem \ref{T1}} \label{S2}

For the \emph{proof of Theorem} \ref{T1}, we arrange the elements $x$ of $X$ 
according to their size $|x|$ and, in case of equal size, in lexicographic 
order of the components. In this way, we obtain a total order $\preceq$ on 
$X$. Since~$X$ is discrete, the number of elements of $X$ having the same 
size is finite. Furthermore, it easily follows that $(X,\preceq)$ is order 
isomorphic to $(\setN,\leq)$, enabling us to argue by induction on $x\in X$.

Rewrite $Tg=0$ as infinite system of equations
\begin{equation} \label{syst}
 \sum_{\substack{y,x'\in X \\ y+x'=x}}\,
 \big(a_d(y)\,g^{*d}(x')+a_{d-1}(y)\,g^{*(d-1)}(x')+\cdots+
 a_0(y)\,u(x')\big)=0
\end{equation}
for $x\in X$. It follows from \eqref{syst} at $x=0$ that necessarily
$f\big(g(0)\big)=0$. Each simple zero $z_0$ of $f(z)$ serves for starting the
following recurrence relation with $g(0)=z_0$.

Now let $0\neq x\in X$. By separating all terms containing $g(x)$ in
\eqref{syst}, we see that the coefficient of $g(x)$ equals the value of the
derivative $f'(z)$ at $z=z_0$. Therefore \eqref{syst} takes the form
\begin{equation} \label{rec}
 f'\big(g(0)\big)\cdot g(x)=
 -\sum_{0\leq j\leq d}~
 \sum_{\substack{y,x_1,\ldots,x_j\in X \\
 y+x_1+\cdots+x_j=x \\ x_1,\ldots,x_j\neq x}}
 a_j(y)\,g(x_1)\cdots g(x_j)
\end{equation}
for $0\neq x\in X$. Due to the choice of $g(0)=z_0$ we have
$f'\big(g(0)\big)\neq 0$. Note that $|x_1|,\ldots,|x_j|<|x|$ in all summands
on the right-hand side in \eqref{rec} so that \eqref{rec} represents a
recursion formula, which uniquely determines an arithmetic function
$g\in\scra$.

To prove the second part of Theorem \ref{T1}, let $q\neq 0$ be an element of 
$X$ of minimal size $|q|>0$. Then there are only two additive decompositions 
of $q=x+x'$ into two summands, namely $q+0$ and $0+q$. Let $a\in\scra$ 
satisfy $a(0)=0\neq a(q)$ and suppose that $Tg=g*g-a$ has a zero $g\in\scra$. 
Then $f(z)=z^2$ vanishes at the double zero $z=0$ only. It follows from
$Tg(0)=0$ that $g(0)=0$ and hence
$Tg(q)=2\,g(0)\,g(q)-a(q)=-a(q)\neq 0$, a contradiction. Therefore $Tg=0$ is 
unsolvable.

Since $\scra$ is an integral domain, the polynomial $Tg$ of degree $d$ has at 
most $d$ zeros $g\in\scra$.
\qed

As an immediate consequence we note

\begin{corollary} \label{C1}
If $f(z)$ satisfies $\deg{f}=d$ and all zeros of $f(z)$ are simple, then,
with the $d$ distinct solutions $g_1,\ldots,g_d\in\scra$ to $Tg=0$, we have
\begin{equation*}
Tg=a_d*(g-g_1)*\cdots*(g-g_d).
\end{equation*}
\end{corollary}


\section{Elementary proof of Theorem \ref{T2}} \label{S3}


The Banach algebras $\scra_\vr$ and $\scra:=\scra_0$ are isomorphic under the
map $a(x)\mapsto e^{-\vr|x|}a(x)$ for $x\in X$, $a\in\scra_\vr$\,. For the
proof of Theorem \ref{T2} we may therefore assume that
$\|a_j\|:=\|a_j\|_0<\infty$. It suffices to show that the solution
$g\in\scra$ to $Tg=0$ with $f(z_0)=0$ and $f'(z_0)\neq 0$ at $z_0=g(0)$
belongs to $\scra_r$ for some $r\geq 0$. With $g_r(x):=e^{-r|x|}g(x)$ we
rewrite \eqref{rec} as
\begin{align*}
f'(z_0)\cdot g_r(x)
 & = - \sum_{0\leq j\leq d}~
     \sum_{\substack{y,x_1,\ldots,x_j\in X \\
     y+x_1+\cdots+x_j=x \\ x_1,\ldots,x_j\neq x}}
     \frac{a_j(y)}{e^{r|y|}}\,g_r(x_1)\cdots g_r(x_j) \\
 & = - \sum_{0\leq j\leq d}\,a_j(0)
     \sum_{\substack{x_1,\ldots,x_j\in X \\
     x_1+\cdots+x_j=x \\ x_1,\ldots,x_j\neq x}}
     g_r(x_1)\cdots g_r(x_j) \\
 & \quad - \sum_{0\leq j\leq d}~
     \sum_{\substack{y,x_1,\ldots,x_j\in X \\ y+x_1+\cdots+x_j=x \\
     y\neq 0,x_1,\ldots,x_j\neq x}}
     e^{-r|y|}\,a_j(y)\,g_r(x_1)\cdots g_r(x_j)
\end{align*}
for $0\neq x\in X$.

We have to prove that there is an $r\geq 0$ such that the partial sums
\begin{equation*}
S_r(m):=\sum_{0<|x|\leq m}\,|g_r(x)|
\end{equation*}
are uniformly bounded for numbers $m$ belonging to the discrete image set
$M=\{m_0,m_1,\ldots\}$ of $X$ under $|\,\,\,|$ with  $m_0=0<m_1<\ldots$, say.
Clearly $S_r(m_0)=0$, and for $n\in\setN$ the above representation of
$f'(z_0)\cdot g_r(x)$ yields

\begin{align*}
|f'(z_0)| & \cdot S_r(m_n)
     \leq \sum_{0\leq j\leq d}\,|a_j(0)|
          \sum_{0<|x|\leq m_n}
          \sum_{\substack{x_1,\ldots,x_j\in X \\
          x_1+\cdots+x_j=x \\ x_1,\ldots,x_j\neq x}}
          |g_r(x_1)|\cdots |g_r(x_j)| \\
 & + \sum_{0\leq j\leq d}~
     \sum_{0<|x|\leq m_n}
     \sum_{\substack{y,x_1,\ldots,x_j\in X \\
     y+x_1+\cdots+x_j=x \\ y\neq 0,x_1,\ldots,x_j\neq x}}
     e^{-r|y|}\,|a_j(y)|\,|g_r(x_1)|\cdots |g_r(x_j)|.
\end{align*}
Let $\Sigma_1$ and $\Sigma_2$ denote the multiple sums on the right-hand
side of this inequality. By extracting all powers of $|g_r(0)|=|g(0)|=|z_0|$
from the inner $j$-fold sum of $\Sigma_1$, we first obtain
\begin{align*}
\Sigma_1
 & \leq \sum_{j=2}^d\,|a_j(0)|\,
        \sum_{i=2}^j\,\binom{j}{i}\,|z_0|^{j-i}
        \sum_{0<|x|\leq m_n}
        \sum_{\substack{x'_1,\ldots,x'_i\in X \\
        x'_1+\cdots+x'_i=x \\ x'_1,\ldots,x'_i\neq 0,x}}
        |g_r(x'_1)|\cdots |g_r(x'_i)| \\
 & \leq \sum_{j=2}^d\,|a_j(0)|~
        \sum_{i=2}^j\,\binom{j}{i}\,|z_0|^{j-i}\,
        S_r^{\,i}(m_{n-1}).
\end{align*}
Next note that $e^{-r|y|}\leq e^{-rm_1}$ for all $y\in X$, $y\neq 0$. Then 
similarly
\begin{align*}
\Sigma_2
 & \leq e^{-rm_1}\,\sum_{0\leq j\leq d}~
        \sum_{0<|x|\leq m_n}~\sum_{\substack{y,x_1,\ldots,x_j\in X \\
        y+x_1+\cdots+x_j=x \\ y\neq 0,x_1,\ldots,x_j\neq x}}
        |a_j(y)|\,|g_r(x_1)|\cdots |g_r(x_j)| \\
 & \leq e^{-rm_1}\,\Bigg(\|a_0\|+\sum_{1\leq j\leq d}\,\|a_j\|
        \sum_{\substack{x_1,\ldots,x_j\in X \\
        |x_1+\cdots+x_j|<m_n}}
        |g_r(x_1)|\cdots |g_r(x_j)|\Bigg) \\
 & \leq e^{-rm_1}\,\sum_{j=0}^d\,\|a_j\|\,\big(|z_0|+S_r(m_{n-1})\big)^j.
\end{align*}
Introduce polynomials $P(t),Q(t)\in\setR[t]$ by
\begin{align*}
P(t) & := \frac{1}{|f'(z_0)|}\,\sum_{j=2}^d\,|a_j(0)|
          \sum_{i=2}^j\,\binom{j}{i}\,|z_0|^{j-i}\,t^i, \\ 
Q(t) & := \frac{1}{|f'(z_0)|}\,\sum_{j=0}^d\,\|a_j\|\,t^j 
\end{align*}
and summarize:
\begin{align}
S_r(m_n)    \label{eq3.3}
 & \leq P\big(S_r(m_{n-1})\big)+e^{-r m_1}\,Q\big(|z_0|+S_r(m_{n-1})\big)
   \qquad(n\in\setN).
\end{align}
This is a recursive estimate starting with $S_r(m_0)=0$. It remains to show
that there exist constants $r\geq 0$ and $C>0$ such that $S_r(m_n)\leq C$ for
all $n\in\setN$.

\noindent
Since $P(t)\geq 0$ and $Q(t)>0$ are increasing functions of $t\in [0,\infty)$ 
with $\deg{P(t)}\leq\deg{Q(t)}=d$ or $P(t)=0$ (null function), it suffices by 
\eqref{eq3.3} to find solutions $r\geq 0$ and $t>0$ of the inequality
\begin{equation*}
P(t)+e^{-r m_1}\,Q(|z_0|+t)\leq t
\end{equation*}
or, equivalently, of
\begin{equation} \label{eq3.4}
e^{-r m_1}\leq\frac{t-P(t)}{Q(|z_0|+t)}=\,:R(t)\,.
\end{equation}
Note that $P(0)=P'(0)=0$ and $R$ is bounded above on $[0,\infty)$. Hence
there exists some $t>0$ with $R(t)>0$. Choosing $r_0>0$ such that
$e^{-r_0 m_1}\leq R(t)$ we obtain \eqref{eq3.4}. Hence $S_{r_0}(m_n)\leq t$
for all $n\in\setN$, which completes the proof. \qed

\begin{remark} \label{R1}
The proof of Theorem \ref{T2} also leads to the quantitative estimates
$\|g\|_r\leq C$ and $r=\vr+m_1^{-1}\max{\big\{0,-\log{C}\big\}}$ with
$C=\sup{\big\{R(t):t\geq 0\big\}}$.
\end{remark}

For $X=\setN_0^k$, Theorem \ref{T2} applies to multidimensional power series
\begin{equation*}
\wt{g}(w)=\sum_{n\in\setN_0^k}\,g(n)\,w^n
\end{equation*}
with coefficient sequences $g\in\scra(X)$ and $w^n=e^{-n\cdot s}$ for
$n\in X$, $s\in\setC^k$. We recover a special case of the implicit function
theorem for complex analytic maps:

\begin{corollary} \label{C2}
Let the power series $\wt{a}_0(w),\ldots,\wt{a}_{d-1}(w)$ and
$\wt{a}_d(w)\neq 0$ be holomorphic functions of $w\in\setC^k$ in a
neighborhood of the origin. Suppose that $z_0\in\setC$ is a simple zero of
the polynomial
\begin{equation*}
a_d(0)\,z^d+a_{d-1}(0)\,z^{d-1}+\cdots+a_0(0)\in\setC[z].
\end{equation*}
Then there exists a local solution $\wt{g}(w)$ with $\wt{g}(0)=z_0$ to
\begin{equation*}
 \wt{a}_d(w)\,\wt{g}^d(w)+\wt{a}_{d-1}(w)\,\wt{g}^{d-1}(w)+\cdots+
 \wt{a}_0(w)=0
\end{equation*}
that is again holomorphic in a neighborhood of the origin.
\end{corollary}

For $X=(\log{\setN})^k$, Theorem \ref{T2} also applies to multidimensional
ordinary Dirichlet series
\begin{equation*}
\wt{g}(s)=\sum_{n\in\setN^k}\,g(n)\,n^{-s}
\end{equation*}
with coefficient sequences $g\in\scra(\setN^k)$ and
$n^{-s}=e^{-(s_1\log{n_1}+\cdots+s_k\log{n_k})}$ for $n\in\setN^k$ and
$s\in\setC^k$ (cf. \cite{GLP2006}, Theorem 3). With
$1:=(1,\ldots,1)\in\setN^k$ we obtain

\begin{corollary} \label{C3}
Let the $k$-dimensional Dirichlet series $\wt{a}_0(s),\ldots,\wt{a}_{d-1}(s)$
and $\wt{a}_d(s)\neq 0$ converge absolutely for all $s$ in some $H_r^k$.
Suppose that $z_0\in\setC$ is a simple zero of the polynomial
\begin{equation*}
a_d(1)\,z^d+a_{d-1}(1)\,z^{d-1}+\cdots+a_0(1)\in\setC[z].
\end{equation*}
Then there exists a Dirichlet series $\wt{g}(s)$ with $g(1)=z_0$ that solves
\begin{equation*}
 \wt{a}_d(s)\,\wt{g}^d(s)+\wt{a}_{d-1}(s)\,\wt{g}^{d-1}(s)+
 \cdots+\wt{a}_0(s)=0
\end{equation*}
and also converges absolutely for all $s$ in some $H_\vr^k$. 
\end{corollary}


\section{Proof of Theorem \ref{T3}} \label{S4}


Our proof of Theorem~\ref{T3} involves general facts concerning analytic 
equations in topological algebras. Recall that a complex \emph{topological 
algebra} is an algebra~$A$ over~$\setC$, equipped with a locally convex 
vector topology making the bilinear algebra multiplication $A\times A\to A$ a 
continuous map. It is called \emph{complete} if the underlying locally convex 
space is complete. A \emph{continuous inverse algebra} is a unital,
associative complex topological algebra~$A$ whose group of units $A^*$ is 
open in~$A$ and whose inversion map $A^*\to A$, $a\mapsto a^{-1}$ is
continuous (see \cite{Bil2004}, \cite{Glo2002} and \cite{Wae1954}). The 
\emph{spectrum} of a commutative continuous inverse algebra $A$ is the set 
$\wh{A}$ of all unital algebra homomorphisms $\xi\colon A\to\setC$. It is 
known that $\xi\mapsto\ker \xi$ is a bijection from $\wh{A}$ onto the set of 
all maximal (proper) ideals of~$A$ (cf. \cite[Lemma~2.5]{Bil2004}). The 
\emph{spectrum} of an element $a\in A$ is defined as 
$\sigma(a):=\{s\in \setC\colon s-a\not\in A^*\}$, and by 
\cite[Theorem~2.7\,(a)]{Bil2004}, it coincides with the set
$\{\xi(a)\colon\xi\in\wh{A}\,\}$. The \emph{joint spectrum} of elements 
$a_1,\ldots,a_n\in A$ is defined as
\[
\sigma(a_1,\ldots,a_n):=
\{(\xi(a_1),\ldots,\xi(a_n))\colon \xi\in \wh{A}\}\subseteq\setC^n.
\]
Then 
$\sigma(a_1,\ldots,a_n)\subseteq\sigma(a_1)\times\cdots\times\sigma(a_n)$.
If $A$ is a commutative, complete continuous inverse algebra,
$a_1,\ldots,a_n\in A$ and $f\colon U\to\setC$ a holomorphic function on an 
open subset $U\subseteq \setC^n$ such that 
$\sigma(a_1,\ldots,a_n)\subseteq U$, then the holomorphic functional
calculus gives rise to an element $f[a_1,\ldots,a_n]\in A$ 
(see \cite[\S4]{Bil2004} for details).

\begin{remark}\label{remfuncalc}
The following simple facts are essential for our purposes:
\begin{itemize}
\item[(a)]
Naturality of the holomorphic functional calculus. If $A$, $f$ and 
$a_1,\ldots,a_n\in A$ are as before and $\phi\colon A\to B$ is a continuous 
homomorphism of unital algebras to a complete, commutative continuous inverse 
algebra $B$, then 
$\sigma(\phi(a_1),\ldots,\phi(a_n))\subseteq \sigma(a_1,\ldots, a_n)$ and
\[
\phi(f[a_1,\ldots,a_n])=f[\phi(a_1),\ldots,\phi(a_n)]
\]
(see \cite[Theorem~4.9]{Bil2004}).
\item[(b)]
If $A=\setC$, then $f[a_1,\ldots,a_n]=f(a_1,\ldots,a_n)$ is the value of $f$ 
at $(a_1,\ldots,a_n)\in U\subseteq \setC^n$.
\end{itemize}
\end{remark}

Our proof of Theorem \ref{T3} uses the following special case of Biller 
\cite[Theorem~8.2]{Bil2004}, applied to algebras whose spectrum is a 
singleton.

\begin{lemma} \label{biller}
Let $A$ be a complete, commutative continuous inverse algebra whose spectrum
is a singleton, $\wh{A}=\{\xi\}$. Let $V\subseteq\setC^n$ and
$Z\subseteq\setC^m$ be open sets and let $F\colon V\times Z\to\setC^m$ be a
holomorphic function. Suppose that $v_0\in V$, $z_0\in Z$ such that
$F(v_0,z_0)=0$ and $\pa_z{F(v_0,z)|_{z=z_0}}\in\GL_m(\setC)$. Then, for each
$(a_1,\ldots, a_n)\in A^n$ satisfying
$\big(\xi(a_1),\ldots,\xi(a_n)\big)=v_0$, there exists a unique
$(g_1,\ldots,g_m)\in A^m$ such that $\big(\xi(g_1),\ldots,\xi(g_m)\big)=z_0$
and
\begin{equation}\label{interpret}
F[a_1,\ldots,a_n,g_1,\ldots, g_m]=0\,.
\end{equation}
\end{lemma}

Given an infinite discrete additive semigroup $X\subseteq [0,\infty)^k$, the 
Dirichlet algebra $\scra=\setC^X$ satisfies the hypotheses of 
Lemma~\ref{biller} when equipped with the product topology.

\begin{lemma}\label{arithCIA}
$\scra$ is a commutative continuous inverse algebra whose spectrum is a
singleton, namely $\wh{\scra}=\{\xi\}$ with
$\xi\colon\scra \to\setC$, $f\mapsto f(0)$.
Furthermore, $\scra$ is a Fr\'{e}chet space $($and
hence complete$)$.
\end{lemma}

\begin{proof}
First note that $\scra^*=\{f\in\scra:f(0)\not=0\}$, which is an open subset 
of $\scra$. To see this, let $f\in\scra$. The function $\xi$ described in the 
lemma is a unital algebra homomorphism to $\setC$. Hence, if $f\in\scra$ and 
$f(0)=\xi(f)=0$, then $f$ is not invertible. If, on the other hand, 
$f(0)\neq 0$, then the equation $f*g=u$ has a unique solution $g$ in $\scra$, 
by Theorem \ref{T1} (and then $g=f^{-1}$). Given $x\in X\setminus\{0\}$, the 
proof of Theorem~\ref{T1} shows that $f^{-1}(x)$ only depends on $f(y)$ for 
$y$ in the finite set $\{y\in X\colon |y|\leq|x|\}$. Moreover, $f^{-1}(x)$ 
is a rational (and hence continuous) function in the $f(y)$. Therefore the 
inversion map $\scra^*\to\scra$, $f\mapsto f^{-1}$ is continuous and thus 
$\scra$ is a continuous inverse algebra. Since 
$\scra\setminus\ker\xi=\scra^*$, it follows that every proper ideal of 
$\scra$ is contained in $\ker{\xi}$. Hence, if $\eta\in\wh{\scra}$, then 
$\ker{\eta}=\ker{\xi}$ (since $\ker{\eta}$ is a maximal ideal) and thus 
$\eta=\xi$. Hence $\wh{\scra}=\{\xi\}$. Being a countably infinite power of 
the Fr\'{e}chet space~$\setC$, also $\scra=\setC^X$ is a Fr\'{e}chet space.
\end{proof}

In connection with Remark \ref{remfuncalc}\,(a), the following lemma will be 
useful.

\begin{lemma}\label{lemmaclear1}
If $X$ is discrete, then the inclusion map $\lambda\colon\scra_r\to\scra$ is 
a continuous algebra homomorphism, for each $r\in\setR$.
\end{lemma}

\begin{proof}
Since~$\scra_r$ is a subalgebra of $\scra$, the inclusion map is an algebra 
homomorphism. Because $\lambda$ is linear and 
$|\lambda(f)(x)|=|f(x)|\leq e^{ry_0x}\|f\|_r$, we see that $\scra_r\to\setC$,
$f\mapsto\lambda(f)(x)$ is continuous for each $x\in X$. Since 
$\scra=\setC^X$ is equipped with the product topology, this implies that 
$\lambda$ is continuous.
\end{proof}

\begin{proof}[Proof of Theorem~\ref{T3}.]
Step~1. By Lemma~\ref{arithCIA}, $\scra$ is a complete, commutative continuous
inverse algebra with spectrum $\{\xi\}$. Since
$\big(\xi(a_1),\ldots,\xi(a_n)\big)=\big(a_1(0),\ldots, a_n(0)\big)=v_0$,
Lemma~\ref{biller} shows the existence and uniqueness of
$(g_1,\ldots,g_m)\in\scra^m$ such that conditions \eqref{cond1} and
\eqref{eq1.5} of Theorem~\ref{T3} are satisfied.

To complete the proof, we shall use Theorem~\ref{Tnew} established below (the 
proof of which is independent of Theorem~\ref{T3}).

Step~2. If $a_1,\ldots,a_n\in \scra_\infty$, then Theorem~\ref{Tnew} 
(combined with Example~\ref{E1}) shows that there is a uniquely determined
$m$-tuple $(g_1,\ldots, g_m)\in\scra_\infty^m$ such that \eqref{cond1} and
\eqref{eq1.5} hold in some~$\scra_r$. The elements 
$(g_1,\ldots, g_m)\in\scra_\infty^m$ coincide with the corresponding elements 
of~$\scra$ obtained in Step~1. To see this, pick $r\in\setR$ with 
$a_1,\ldots,a_n,g_1,\cdots,g_m\in\scra_r$ and 
$F[a_1,\ldots,a_n,g_1,\ldots,g_m]=0$ in $\scra_r$. Since the inclusion map 
$\lambda\colon\scra_r\to\scra$ is a continuous algebra homomorphism (by 
Lemma~\ref{lemmaclear1}), we obtain
\begin{equation*}
 0=\lambda\big(F[a_1,\ldots, a_n,g_1,\ldots,g_m]\big)
  =F[\lambda(a_1),\ldots,\lambda(a_n),\lambda(g_1),\ldots,
   \lambda(g_m)]
\end{equation*}
due to the naturality of holomorphic functional calculus (see 
Remark~\ref{remfuncalc}\,(a)). Now the uniqueness assertion from Step~1
applies.

Step~3. In view of \eqref{corresp}, the validity of \eqref{condpre2} and 
\eqref{cond2} for large~$r$ follows from \eqref{cond1new} and 
\eqref{cond2new} in Theorem~\ref{Tnew}.

Step~4.
In Theorem~\ref{Tnew}, \eqref{cond0new}, \eqref{cond1new} and 
\eqref{cond2new} imply \eqref{cond3new}. Hence, as a special case,
\eqref{cond1}, \eqref{condpre2} and \eqref{cond2} imply \eqref{eq1.5} and 
thus determine $(g_1,\ldots,g_m)$.
\end{proof}


\section{Technical preliminaries} \label{S5}


The measures required for our purposes are (possibly unbounded\,!) complex 
Radon measures. Since a suitable reference describing the relevant aspects of 
their theory does not seem to be available, we add this section for the 
convenience of readers with a standard knowledge of measure theory. Various 
results on Laplace transforms are also provided. Our main sources are 
\cite{BCR1984}, \cite{Bou1969}, and \cite{FD1988}.

\emph{The measure-theoretic setting}.
Given a Hausdorff topological space~$X$, let $\delta(X)$ be the $\delta$-ring 
generated by the set $\scrk(X)$ of compact subsets of~$X$ (thus $\delta(X)$ 
is the smallest set containing $\scrk(X)$ and closed under finite unions, 
relative complements and countable intersections). A function 
\mbox{$\mu\colon\delta(X)\to\setC$} is called a \emph{complex measure} if
$\mu(B)=\sum_{n=1}^\infty\mu(B_n)$ for all sequences 
$(B_n)_{n\in \setN}$ of disjoint sets $B_n\in \delta(X)$ such that 
$B:=\bigcup_{n\in\setN}B_n\in\delta(X)$. If, furthermore, 
$\mu(B)\in[0,\infty)$ for each $B\in\delta(X)$, then $\mu$ is called a 
\emph{positive measure}. To any \mbox{complex} measure $\mu$,
\cite[Proposition II.1.3]{FD1988} associates a positive measure $|\mu|$
(the total variation measure).
A complex measure $\mu$ is called a \emph{Radon measure} if $|\mu|$ is inner 
regular, i.e., $|\mu|(B)=\sup_{K\in\scrk(B)}|\mu|(K)$. If $\mu$ is a 
positive Radon measure on~$X$, then $\mu|_{\scrk(X)}$ is a Radon content in 
the sense of \cite[Definition~2.1.2]{BCR1984} and hence extends uniquely to 
an inner regular measure $\wb{\mu}\colon\scrb(X)\to[0,\infty]$ on the Borel 
$\sigma$-algebra $\scrb(X)$ of~$X$ (see \cite[Theorem~2.1.4]{BCR1984}). By 
abuse of notation, we shall frequently write~$\mu$ in place of~$\wb{\mu}$.

\begin{remark}\label{patch}
Note that every measure $\mu$ gives rise to a family
$(\mu_K)_{K\in\scrk(X)}$ of measures $\mu_K:=\mu|_{\scrb(K)}$ which are 
compatible in the sense that $\mu_L|_{\scrb(K)}=\mu_K$ for all 
$K,L\in\scrk(X)$ with $K\subseteq L$. Here~$\mu$ is positive if and only if 
each $\mu_K$ is positive. Since
\begin{equation}\label{dirunion}
\delta(X)= \bigcup_{K\in \scrk(X)}\scrb(K)\,,
\end{equation}
it is easy to see that, conversely, every compatible family
$(\mu_K)_{K\in\scrk(X)}$ defines a complex Radon measure~$\mu$ via 
$\mu|_{\scrb(K)}:=\mu_K$ for $K\in \scrk(X)$.
\end{remark}

\noindent
If $\mu$ is a complex Radon measure on~$X$, then $\mu_K$ is a bounded Radon 
measure for each $K\in \scrk(X)$ and thus 
$\mu_K=i\mu_K^1-\mu_K^2-i\mu_K^3+\mu_K^4$ with positive Radon measures 
$\mu_K^j$, $j\in \{1,\ldots, 4\}$, where $\mu_K^4$ and $\mu_K^2$ are the 
positive and negative variations of the real part of~$\mu_K$, respectively, 
and $\mu_K^1$, $\mu_K^3$ are those of its imaginary part 
(see \cite[\S6.6]{Rud1987}). Since $\mu_K^j\leq |\mu|_K$, the measure 
$\mu_K^j$ has a density with respect to the Radon measure $|\mu|_K$ (by the 
Radon-Nikodym theorem), entailing that~$\mu_K^j$ is inner regular. By 
Remark~\ref{patch}, the families $(\mu_K^j)_{K\in\scrk(X)}$ determine 
positive Radon measures~$\mu_j$ on $X$ for $j\in\{1,\ldots,4\}$ such 
that $\mu_j\leq |\mu|$ and
\begin{equation}\label{Jordan}
\mu = i\mu_1-\mu_2-i\mu_3+\mu_4\,.
\end{equation}
We say that a Borel measurable function $f\colon X\to\setC$ is 
\emph{$\mu$-integrable} if $f\in L^1(X,|\mu|)$. In this case, we write
$f=\sum_{j=1}^4 i^j f_j$ with $0\leq f_j\in L^1(X,|\mu|)$ and 
$\mu=\sum_{k=1}^4i^k\mu_k$ with positive Radon measures $\mu_k$ such that
$\mu_k\leq|\mu|$ and define%
\footnote{Typically, $f_4$ and $f_2$ (resp., $f_1$ and $f_3$) are the 
positive and negative parts of the real part (resp., imaginary part) 
of~$f$, and $\mu_1,\ldots,\mu_4$ are as in \eqref{Jordan}.}
\[
\int_Xf\,d\mu := \sum_{j,k=1}^4i^{j+k}\int_Xf_j\,d\mu_k\,.
\]
Let $M(X)$ be the space of complex Radon measures on~$X$, let 
$M_0(X):=\{\mu\in M(X):\|\mu\|:=|\mu|(X)<\infty\}$ be the space of bounded 
complex Radon measures, and $M_+(X)$ be the set of positive Radon measures. 
Then $(M_0(X),\|\cdot\|)$ is a Banach space, because the space of all
bounded complex measures on~$X$ is a Banach space (see 
\cite[II.1.5]{FD1988}) and also $M(K)$ is a Banach space for each 
$K\in\scrk(X)$ (entailing, in view of Remark~\ref{patch}, that limits of 
complex Radon measures are again Radon).

If $\mu\in M_0(X)$ in \eqref{Jordan}, then 
$\mu_j\in M_{0,+}(X):=M_0(X)\cap M_+(X)$ for each $j\in \{1,\ldots,4\}$, and 
$\wb{\mu}:=i\,\wb{\mu}_1-\wb{\mu}_2-i\,\wb{\mu}_3+\wb{\mu}_4$ is the unique 
extension of~$\mu$ to an (ordinary) complex measure on $(X,\scrb(X))$ whose 
total variation is inner regular. Again, we usually write $\mu$ instead of 
$\wb{\mu}$.

Given Hausdorff spaces $X_j$ for $j\in\{1,2\}$ and $\mu_j\in M_+(X)$, there 
exists a unique positive Radon measure $\mu_1\otimes\mu_2$ on $X_1\times X_2$ 
such that
\begin{align}
(\mu_1\otimes\mu_2)(B_1\times B_2) = \mu_1(B_1)\mu_2(B_2)\,
 &   \text{ for all $B_1\in \scrb(X_1)$}\label{charprodm} \\
 &   \text{ and $B_2\in \scrb(X_2)$}\notag
\end{align}
(see \cite[Corollary~2.1.11]{BCR1984}). Since 
$(\mu_1+t \nu_1)\otimes\mu_2=\mu_1\otimes \mu_2+t (\nu_1\otimes\mu_2)$ and 
$\mu_1\otimes(\mu_2+t \nu_2)=\mu_1\otimes\mu_2+t(\mu_1\otimes\nu_2)$ for all 
$\mu_j,\nu_j\in M_{0,+}(X_j)$ and $t\geq 0$, where $M_0(X_j)$ is spanned by 
$M_{0,+}(X_j)$ as a complex vector space, standard arguments provide a unique 
complex bilinear map
\[
\beta\colon M_0(X_1)\times M_0(X_2)\to M_0(X_1\times X_2)
\]
such that $\beta(\mu_1,\mu_2)=\mu_1\otimes\mu_2$ for all 
$\mu_j\in M_{0,+}(X_j)$. We write $\mu_1\otimes\mu_2:=\beta(\mu_1,\mu_2)$ 
also for general $\mu_j\in M_0(X_j)$. Then
\begin{equation}\label{newthing}
|\mu_1\otimes\mu_2|\leq |\mu_1|\otimes|\mu_2|\quad
\mbox{for all $\,\mu_1\in M_0(X_1)$, $\mu_2\in M_0(X_2)$.}
\end{equation}
To see this, note first that
\begin{equation}\label{newthing2}
(\rho_1\odot\mu_1)\otimes (\rho_2\odot \mu_2) =
(\rho_1\otimes \rho_2)\odot (\mu_1\otimes\mu_2)
\end{equation}
for all $\mu_j\in M_+(X_j)$ and Borel measurable functions
$\rho_j\colon X_j\to[0,\infty]$ such that $\int_K\rho_j\,d\mu_j<\infty$ for 
each $K\in\scrk(X_j)$ for $j\in\{1,2\}$, where
\[
\rho_1\otimes\rho_2\colon X_1\times X_2\to\setC\,,\quad
(x_1,x_2)\mapsto\rho_1(x_1)\rho_2(x_2)\,.
\]
In fact, the right hand side of \eqref{newthing2} is a positive Radon 
measure which satisfies the characterization of the product measure on the 
left (cf. \eqref{charprodm}).

As a consequence of the Radon-Nikodym theorem, $\mu_j$ as in \eqref{newthing} 
admits a polar decomposition $\mu_j=\rho_j\odot|\mu_j|$, for a suitable 
measurable function $\rho_j\colon X_j\to \setC$ such that $|\rho_j(x)|=1$ for 
each $x\in X_j$ (see \cite[Theorem~6.12]{Rud1987}). We write
$\rho_j=\sum_{k=1}^4i^k\rho_j^k$ with $\rho_j^1:=\imp(\rho_j)_+$,
$\rho_j^2:=\re(\rho_j)_-$, $\rho_j^3:=\imp(\rho_j)_-$ and
$\rho_j^4:=\re(\rho_j)_+$. Then, by \eqref{newthing2} and the
bilinearity of $\otimes$ and $\odot$,
\begin{align*}
\mu_1\otimes\mu_2
 & = \sum_{k,\ell=1}^4i^{k+\ell}(\rho_1^k\odot|\mu_1|)\otimes
     (\rho_2^\ell\odot|\mu_2|) \\
 & = \sum_{k,\ell=1}^4i^{k+\ell}(\rho_1^k\otimes\rho_2^\ell)\odot
     (|\mu_1|\otimes |\mu_2|) \\
 & = (\rho_1\otimes\rho_2)\odot (|\mu_1|\otimes|\mu_2|)\,,
\end{align*}
entailing that $|\mu_1\otimes\mu_2|=|\mu_1|\otimes|\mu_2|$.
Thus \eqref{newthing} holds.

We now return to the situation described in Section \ref{S1}, where 
$C\subseteq\setR^k$ is a pointed and generating closed convex cone and 
$X\subseteq C$ a (continuously embedded) topological semigroup with 
continuous addition
\begin{equation}\label{addmap}
\alpha\colon X\times X\to X\,,\quad
(x_1,x_2)\mapsto x_1+x_2\,.
\end{equation}
It is useful to observe that each $\mu\in M_{r,+}(X):=M_r(X)\cap M_+(X)$ (with 
$M_r(X)$, $y_0$ and $\|\cdot\|_r$ as in the introduction) is a $\sigma$-finite 
measure, since
\[
\int_X e^{-rxy_0}\,d\mu(x)=\|\mu\|_r<\infty\,,
\]
where $e^{-rxy_0}>0$ for each $x\in X$. Given $\mu_1,\mu_2\in M_+(X)$, we 
define their \emph{convolution} as the image measure
\[
\mu_1*\mu_2 := \alpha(\mu_1\otimes\mu_2)
\]
on $\scrb(X)$; thus $(\mu_1*\mu_2)(B):=(\mu_1\otimes\mu_2)(\alpha^{-1}(B))$. 
If $\mu_1*\mu_2$ is finite on compact sets, then $\mu_1*\mu_2$ is a positive 
Radon measure~\cite[Proposition~2.1.15]{BCR1984}. In particular, 
$\mu_1*\mu_2$ is a positive Radon measure if $\mu_1,\mu_2\in M_{r,+}(X)$ for 
some $r\in\setR$. In fact, for each compact set $K\subseteq X$, we have 
$a:=\inf\{e^{-rxy_0}\colon x\in K\}>0$ and
\begin{align}
a\mu(K) 
 & \leq \int_X e^{-rxy_0}\,d(\mu_1*\mu_2)(x)=
        \int_X e^{-rxy_0}\,d\alpha(\mu_1\otimes\mu_2)(x) \notag \\
 & =    \int_{X\times X} e^{-r\alpha(x_1,x_2)y_0}\,d(\mu_1\otimes \mu_2)
        (x_1,x_2)\notag \\
 & =    \int_{X\times X} e^{-rx_1y_0}\,e^{-rx_2y_0}\,d(\mu_1\otimes \mu_2)
        (x_1,x_2) \notag \\
 & =    \int_X e^{-rx_1y_0}\,d\mu_1(x_1) \int_X e^{-rx_2y_0}\,d\mu_2(x_2)<
        \infty \label{reusefubini}
\end{align}
(by Transformation of Integrals and Fubini's Theorem),%
\footnote{Since $\mu_1,\mu_2$ are $\sigma$-finite measures and the map 
$X\times X\to\setC$, $(x_1,x_2)\mapsto e^{-rx_1y_0} e^{-rx_2y_0}$ is 
$\scrb(X)\otimes\scrb(X)$-measurable, the standard Fubini theorem (as in 
\cite[Theorem~8.12]{Rud1987}) suffices; we do not need specialized versions 
for Radon measures like \cite[Theorem~2.1.12]{BCR1984}.} 
whence $\mu(K)<\infty$. The previous calculation also shows that
\begin{equation}\label{pospieces}
\mu_1*\mu_2\in M_r(X)\quad\mbox{for all $\mu_1,\mu_2\in
M_{r,+}(X)$.}
\end{equation}
Since $(\mu_1+t \nu_1)*\mu_2=\mu_1*\mu_2+t (\nu_1*\mu_2)$ and
$\mu_1*(\mu_2+t\nu_2)=\mu_1*\mu_2+t(\mu_1*\nu_2)$ for all
$\mu_j,\nu_j\in M_{r,+}(X)$ and $t\geq 0$, where $M_r(X)$ is spanned by 
$M_{r,+}(X)$ as a complex vector space, standard arguments provide a unique 
complex bilinear map
\begin{equation}\label{firstconv}
M_r(X)\times M_r(X)\to M_r(X)\,,\quad (\mu_1,\mu_2)\mapsto \mu_1*\mu_2
\end{equation}
extending the convolution of measures in $M_{r,+}(X)$ already defined.

If $\mu_1,\mu_2\in M_0(X)$, then
\begin{equation}\label{altconv}
\mu_1*\mu_2 = \alpha(\mu_1\otimes\mu_2)\,.
\end{equation}
In fact, if we re-define convolution via \eqref{altconv}, then
$|\mu_1*\mu_2|\leq |\mu_1|*|\mu_2|$ (as a consequence of \eqref{newthing}) 
and thus $\mu_1*\mu_2\in M_0(X)$ with
\begin{equation}\label{henceba}
\|\mu_1*\mu_2\| \leq (|\mu_1|*|\mu_2|)(X) = \|\mu_1\|\cdot \|\mu_2\|\,.
\end{equation}
Since $\mu_1*\mu_2$ from \eqref{altconv} coincides with the old definition in 
the case of positive measures and is bilinear in $(\mu_1,\mu_2)$, it 
coincides with the convolution defined in~\eqref{firstconv}. It is clear from 
\eqref{altconv} that convolution is associative and commutative (because so
is~$\alpha$). Since also \eqref{henceba} holds and $\|\delta_0\|=1$, we see 
that $(M_0(X),*,\|\cdot\|)$ is a unital commutative Banach algebra. 
For each $r\in\setR$, the map
\[
\phi\colon M_r(X)\to M_0(X)\,,\quad \mu\mapsto e^{-rxy_0}\,d\mu(x)
\]
is a surjective linear isometry and an isomorphism of algebras. Hence also 
$(M_r(X),*,\|\cdot\|_r)$ is a unital commutative Banach algebra.

To complete our discussion of measures, let us show that
$M_\infty(X)$ is indeed independent of the choice of~$y_0$ (as
claimed in the introduction).

\begin{lemma}
If $y_0,y_1\in (C^\star)^0$, then
$\bigcup_{r\in \setR}M_r^{(y_0)}(X)=\bigcup_{r\in \setR}M_r^{(y_1)}(X)$.
\end{lemma}

\begin{proof}
Since $y_1\in (C^\star)^0$, there exists $\ve>0$ such that
$c:=y_1-\ve y_0\in C^\star$ and thus $y_1=\ve y_0 +c$. Then
$rxy_1=r \ve x y_0 +r xc\geq r \ve x y_0$ for each $r\geq 0$ and $x\in X$, 
whence $\|\mu\|_r^{(y_1)}\leq \|\mu\|_{r\ve}^{(y_0)}$ for each 
$\mu\in M^{(y_0)}_{r\ve}(X)$ and thus 
$M_{r\ve}^{(y_0)}(X)\subseteq M_r^{(y_1)}(X)$. Hence
$\bigcup_{r\in \setR}M_r^{(y_0)}(X)\subseteq
\bigcup_{r\in \setR}M_r^{(y_1)}(X)$. The opposite inclusion can be shown 
analogously.
\end{proof}

\noindent\emph{Some basic facts concerning Laplace transforms}.
Fourier-Laplace transforms have the following properties (part of which will 
be essential later).

\begin{lemma}\label{propslapl}
Let $C\subseteq \setR^k$ be a pointed and generating closed convex cone, 
$X\subseteq C$ be a continuously embedded topological semigroup with 
$0\in X$, and $y_0\in (C^\star)^\circ$. Let $r\in \setR$ and 
$\mu\in M_r(X)$. Then the following holds:
\begin{itemize}
\item[\rm(a)]
The function $\wt{\mu}$ is holomorphic on the interior of~$D_\mu$.
\item[\rm(b)]
On $ry_0+C^\star+i\setR^k$, the function~$\wt{\mu}$ is continuous.
\item[\rm(c)] 
For each $\ve>0$, there is $\rho\in [r,\infty)$ such that
\[
|\wt{\mu}(s)-\wt{\mu}(\infty)|\leq \ve\quad
\mbox{for each $\,s\in \rho y_0+C^\star+i\setR^k$.}
\]
\item[\rm(d)]
If also $\nu\in M_r(X)$ and $\wt{\mu}|_U=\wt{\nu}|_U$ for some non-empty open 
set $U\subseteq\setC^k$ $($or $U\subseteq\setR^k)$ such that 
$U\subseteq D_\mu\cap D_\nu$, then $\mu=\nu$.
\end{itemize}
\end{lemma}

\begin{proof}
(a) Write $\mu=i\mu_1-\mu_2-i\mu_3+\mu_4$ as a linear combination of positive 
measures, as in~\eqref{Jordan}. Then 
$\wt{\mu}=i\wt{\mu}_1-\wt{\mu}_2-i\wt{\mu}_3+\wt{\mu}_4$ on $D_\mu^0$, where 
$D_\mu^0\subseteq D_{\mu_j}^0$ and $\wt{\mu}_j$ is holomorphic on 
$D_{\mu_j}^0$ for each $j\in\{1,2,3,4\}$, by 
\cite[Proposition~V.4.6]{Nee2000}.

(b) Let $(s_n)_{n\in \setN}$ be convergent sequence in
$ry_0+C^\star+i\setR^k$, with limit~$s$. Then $\wt{\mu}(s_n)\to\wt{\mu}(s)$ 
as $n\to\infty$ by Lebesgue's dominated convergence theorem, using the 
integrable majorant $X\to[0,\infty)$, $x\mapsto e^{-rxy_0}$.

(c) If $s_n=\rho_n+t_n$ with $t_n\in C^\star+i\setR^k$ and
$\rho_n\in [r,\infty)$ such that $\rho_n\to\infty$, then
$e^{-s_nx}\to {\bf 1}_{\{0\}}(x)$ (with ${\bf 1}_{\{0\}}\colon X\to\{0,1\}$ 
the characteristic function of~$\{0\}$) because 
$|e^{-s_nx}|\leq e^{-\rho_n x y_0}$, where $xy_0>0$ for 
$x\in X\setminus\{0\}$. Hence 
$\wt{\mu}(s_n)\to\int_X {\bf 1}_{\{0\}}(x)\,d\mu(x)=\mu(\{0\})=
\wt{\mu}(\infty)$, by dominated convergence with majorant $e^{-rxy_0}$.

(d) Suppose first that $\mu$ and $\nu$ are positive measures. Let
$\lambda\colon X\to\setR^k$ be the inclusion map and $\mu_1:=\lambda(\mu)$, 
$\nu_1:=\lambda(\nu)$ be the image measures on $\setR^k$ (equipped with the 
Borel $\sigma$-algebra). Then $\mu_1=\nu_1$ (for example, by
\cite[Theorem~14.11\,(e)]{Glo2003}) and 
$\mu|_{\scrk(X)}=\mu_1|_{\scrk(X)}=\nu_1|_{\scrk(X)}=\nu|_{\scrk(X)}$. 
Consequently $\mu=\nu$.

The general case amounts to injectivity of the linear map $M_r(X)\to\setC^U$, 
$\mu\mapsto \wt{\mu}|_U$, which we prove using an idea from \cite[proof of 
Proposition 6.5.2]{BCR1984}. Suppose that $\mu\in M_r(X)$ and 
$\wt{\mu}|_U=0$. We write
\[
\mu = i\mu_1-\mu_2-i\mu_3+\mu_4
\]
with positive measures $\mu_1,\ldots,\mu_4\in M_r(X)$, as in \eqref{Jordan}. 
Then $\wt{\mu}|_{D_\mu^0}=0$ by the Identity Theorem for analytic functions, 
and thus
\begin{equation}\label{firstlinc}
i\,\wt{\mu}_1(s)-\wt{\mu}_2(s)-i\,\wt{\mu}_3(s)+\wt{\mu}_4(s) = 0
\quad\text{for all $s\in D_\mu^0$.}
\end{equation}
Since $D_\mu^0$ is invariant under the complex conjugation 
$\overline{\phantom{c}}$\,, we see that also 
$0=\overline{\wt{\mu}(s)}=\wt{\overline{\mu}}(\overline{s})$ for all 
$s\in D_\mu^0$ and hence
\begin{equation}\label{secondlinc}
 -i\,\wt{\mu}_1(s)-\wt{\mu}_2(s)+i\,\wt{\mu}_3(s)+\wt{\mu}_4(s)=0 \quad
 \text{for all $s\in D_\mu^0$.}
\end{equation}
Adding \eqref{firstlinc} and \eqref{secondlinc}, we deduce that
$\wt{\mu}_2(s)=\wt{\mu}_4(s)$ and hence also
$\wt{\mu}_1(s)=\wt{\mu}_3(s)$, for all $s\in D_\mu^0$. Then
$\mu_1=\mu_3$ and $\mu_2=\mu_4$ by the case of positive measures already 
discussed, and hence $\mu=0$.
\end{proof}

The next lemma enables us to use the naturality of holomorphic functional 
calculus (see Remark~\ref{remfuncalc}\,(a)) in connection with Laplace 
transforms.

\begin{lemma}\label{lemmaclear2}
For each $s\in ry_0 +C^\star+i\setR^k$, the map
\[
 h_s\colon M_r(X) \to\setC\,, \quad
 h_s(\mu):=\int_X e^{-sx}\,d\mu(x)
\]
is a continuous algebra homomorphism. Furthermore,
\begin{equation}\label{istklar}
h_s(\mu)=\wt{\mu}(s)\,.
\end{equation}
\end{lemma}

\begin{proof}
It is clear that $h_s$ is linear. Since 
$|h_s(\mu)|\leq \int_X e^{-ry_0x} d|\mu|(x)=\|\mu\|_r$, the linear map $h_s$ 
is continuous. Since $h_s$ is linear and $M_r(X)$ is spanned by positive 
measures, $h_s(\mu_1*\mu_2)=h_2(\mu_1)h_s(\mu_2)$ will hold for arbitrary 
$\mu_1,\mu_2\in M_r(X)$ if we can prove it for positive measures 
$\mu_1,\mu_2$. In the latter case, repeating the calculation leading to 
\eqref{reusefubini}, we obtain
\begin{align*}
h_s(\mu_1*\mu_2) 
 & = \int_X e^{-sx}\,d(\mu_1*\mu_2)(x) \\
 & = \int_X e^{-sx_1}\,d\mu_1(x_1) \int_X e^{-sx_2}\,d\mu_2(x_2) =
     h_s(\mu_1)h_s(\mu_2)\,.
\end{align*}
Since also $h_s(\delta_0)=1$, the map $h_s$ is a homomorphism of unital 
algebras.
\end{proof}

The following consequence will be applied repeatedly.

\begin{lemma}\label{innocuous}
Let $\mu_1,\ldots,\mu_n\in M_r(X)$ and $f\colon \setC^n\supseteq U\to\setC$ 
be a holomorphic function such that $\nu:=f[\mu_1,\ldots,\mu_n]$ is defined 
in $M_r(X)$. Then
\[
 \wt{\nu}(s) = f(\wt{\mu}_1(s),\ldots,\wt{\mu}_n(s))\,,\quad
 \text{for all $\,s\in ry_0+C^\star+i\setR^k$.}
\]
\end{lemma}

\begin{proof}
Using Lemma~\ref{lemmaclear2} and the facts compiled in 
Remark~\ref{remfuncalc}, we obtain $\wt{\nu}(s)=h_s(\nu)=
h_s(f[\mu_1,\ldots,\mu_n])=f[h_s(\mu_1),\ldots, h_s(\mu_n)]
=f(\wt{\mu}_1(s),\ldots,\wt{\mu}_n(s))$.
\end{proof}


\section{Proof of Theorem \ref{Tnew}} \label{S6}


The proof of Theorem \ref{Tnew} relies on the following lemma (where $A_r$ is 
as in the theorem). Given $\mu\in A_r$ and $t\geq r$, we write 
$\sigma_t(\mu)$ for the spectrum of~$\mu$ in the commutative Banach algebra 
$A_t$ and, likewise, $\sigma_t(\mu_1,\ldots,\mu_n)$ for joint spectra 
in $A_t$. Given $\ve>0$ and $z\in \setC$, let 
$B_\ve(z):=\{w\in\setC:|w-z|<\ve\}$.

\begin{lemma}\label{Lnew}
If $\mu\in A_r$, then
\[
\|\mu-\wt{\mu}(\infty)\delta_0\|_t\to 0\quad\text{as $t\to\infty$.}
\]
For every $\ve>0$, there exists $t_0\geq r$ such that
$\sigma_t(\mu)\subseteq B_\ve(\wt{\mu}(\infty))$ for all $t\geq t_0$.
\end{lemma}

\begin{proof}
If $\|\mu-\wt{\mu}(\infty)\delta_0\|_t<\ve$, then
$\sigma_t(\mu-\wt{\mu}(\infty)\delta_0)\subseteq B_\ve(0)$ (see 
\cite[Corollary 3 to Theorem 18.4]{Rud1987}) and thus
$\sigma_t(\mu)\subseteq B_\ve(\wt{\mu}(\infty))$. Hence the second assertion 
follows if we can prove the first. To this end, let $(t_n)_{n\in \setN}$ be 
a sequence of real numbers $t_n\geq r$ such that $t_n\to\infty$ as 
$n\to\infty$. Since 
$|\mu-\wt{\mu}(\infty)\delta_0| =|\mu|-|\mu|(\{0\})\delta_0$, Lebesgue's 
Dominated Convergence Theorem shows that
\begin{align*}
\|\mu-\wt{\mu}(\infty)\delta_0\|_{t_n}
 & =   \int_X e^{-t_nxy_0}\,d|\mu|(x)-
       \int_X e^{-t_nxy_0}\,|\mu|(\{0\})\,d\delta_0(x) \\
 & =   \int_{X\setminus \{0\}} e^{-t_nxy_0}\,d|\mu|(x)\\
 & \to 0 \quad \text{as $n\to\infty$,}
\end{align*}
using that $e^{-t_nxy_0}\to 0$ as $n\to\infty$ for each 
$x\in X\setminus\{0\}$ and $e^{-t_nxy_0}\leq e^{-rxy_0}$, where 
$\int_{X\setminus\{0\}}e^{-rxy_0}\,d|\mu|(x)\leq\|\mu\|_r<\infty$.
\end{proof}

\begin{proof}[Proof of Theorem {\rm\ref{Tnew}}]
Existence of $\mu$:
By the Implicit Function Theorem, there exist open neighborhoods 
$V_0\subseteq V$ of~$v_0$ and $Z_0\subseteq Z$ of~$z_0$ such that
\begin{equation}\label{Enew1}
\{(v,z)\in V_0\times Z_0:F(v,z)=0\}=\text{graph}(\phi)
\end{equation}
is the graph of a holomorphic function
\[
\phi=(\phi_1,\ldots,\phi_m)\colon V_0\to Z_0
\]
with $\phi(v_0)=z_0$. After shrinking $V_0$, we may assume that
\[
\phi_1(V_0)\times\cdots\times \phi_m(V_0)\subseteq Z_0\,.
\]
Let $r\in\setR$ be such that $\lambda_1,\ldots,\lambda_n\in A_r$. Since 
$V_0$ is an open neighborhood of 
$(\wt{\lambda}_1(\infty),\ldots,\wt{\lambda}_n(\infty))$,
Lemma~\ref{Lnew} implies that, after increasing~$r$ if necessary, we have
\[
\sigma_t(\lambda_1)\times\cdots\times \sigma_t(\lambda_n)\subseteq V_0\quad
\text{for all $\,t\geq r$.}
\]
Hence
\[
\mu_j := \phi_j[\lambda_1,\ldots,\lambda_n]\in A_r
\]
can be defined using functional calculus in~$A_r$, for $j\in \{1,\ldots,m\}$.
Set $\mu:=(\mu_1,\ldots,\mu_m)$. By the Spectral Mapping Theorem
\cite[Corollary~4.10]{Bil2004},
\[
 \sigma_r(\mu_j)=\phi_j(\sigma_r(\lambda_1,\ldots,\lambda_n))
 \subseteq\phi_j(V_0)\,.
\]
Then $\sigma_r(\lambda_1,\ldots,\lambda_n,\mu_1,\ldots,\mu_m)\subseteq 
\sigma_r(\lambda_1)\times\cdots\times\sigma_r(\lambda_n)\times
\sigma_r(\mu_1)\times\cdots\times\sigma_r(\mu_m)$ $\subseteq V_0\times Z_0$.
Moreover,
\begin{equation}\label{Enew2}
F[\lambda_1,\ldots,\lambda_n,\mu_1,\ldots,\mu_m]
= \bigl(F\circ (\id_{V_0},\phi)\bigr)[\lambda_1,\ldots,
\lambda_n]\,=\, 0
\end{equation}
by the Spectral Mapping Theorem \cite[Corollary~4.10]{Bil2004}, since 
$F\circ (\id_{V_0},\phi)=0$ by (\ref{Enew1}).

\noindent
Using Lemma~\ref{innocuous}, we deduce from \eqref{Enew2} that
\[
0 = F[\lambda_1,\ldots,\lambda_n,\mu_1,\ldots,\mu_m]\upwt(s) =
    F(\wt{\lambda}_1(s),\ldots,\wt{\lambda}_n(s),\wt{\mu}_1(s),\ldots,
    \wt{\mu}_m(s)),
\]
for each $s\in ry_0+C^\star+i\setR^k$. Thus \eqref{cond1new}, 
\eqref{cond2new} from Theorem~\ref{Tnew} hold. In view of \eqref{Enew1}, the 
preceding equality also implies \eqref{cond0new}.

Uniqueness of~$\mu$: Let $\nu=(\nu_1,\ldots,\nu_m)\in (A_t)^m$ for some 
$t\in\setR$ such that \eqref{cond0new}, \eqref{cond1new} and \eqref{cond2new} 
hold, with $\mu$ replaced by~$\nu$ and~$r$ replaced by~$t$. After 
increasing~$r$ or~$t$, we may assume that $r=t$. After increasing~$r$ further 
if necessary, we can achieve (using Lemma~\ref{Lnew}) that
\begin{equation}\label{Enew4}
\sigma_r(\nu_1)\times\cdots\times\sigma_r(\nu_m)\subseteq Z_0.
\end{equation}
Hence
\[
\zeta := F[\lambda_1,\ldots,\lambda_n,\nu_1,\ldots,\nu_m]
\]
can be defined in~$A_r$. For each $s\in ry_0+C^\star+i\setR^k$, we
have
\begin{equation}\label{Enew5}
 \wt{\zeta}(s) =
 F(\wt{\lambda}_1(s),\ldots,\wt{\lambda}_n(s),\wt{\nu}_1(s),\ldots,
 \wt{\nu}_m(s))= 0\,,
\end{equation}
by hypothesis and Lemma~\ref{innocuous}. In view of \eqref{Enew1}, 
\eqref{Enew4} and \eqref{Enew5}, we have
$(\wt{\nu}_1(s),\ldots,\wt{\nu}_m(s))=\phi(s)
=(\wt{\mu}_1(s),\ldots,\wt{\mu}_m(s))$ for each $s\in ry_0+C^\star+i\setR^k$. 
Hence $\nu=\mu$, by Lemma~\ref{propslapl}\,(d).

Proof of the final assertion: Assume that
$(\wt{\nu}_1(\infty),\ldots,\wt{\nu}_m(\infty))=z_0$ and
\begin{equation}\label{Enew3}
F[\lambda_1,\ldots,\lambda_n,\nu_1,\ldots,\nu_m] = 0
\end{equation}
in some~$A_t$. After increasing~$r$ or~$t$, we may again assume that $t=r$. 
Applying $h_s$ to \eqref{Enew3} for $s\in ry_0+C^\star+i\setR^k$, we see that 
\eqref{cond1new} and \eqref{cond2new} hold. Hence $\nu=\mu$, by what has just 
been shown.
\end{proof}



\end{document}